\documentclass[a4paper,11pt]{article}

\usepackage{amsmath}
\usepackage{amssymb}
\usepackage{amsthm}
\usepackage[dvipdfmx]{graphicx}
\usepackage{amscd}
\usepackage{comment}

\newcommand{\supp}{\mathop{\mathrm{supp}}\nolimits}
\newcommand{\Fix}{\mathop{\mathrm{Fix}}\nolimits}
\newcommand{\cl}{\mathop{\mathrm{cl}}\nolimits}

\title{Higher dimensional Thompson groups have Serre's property FA}

\author{Motoko Kato \thanks{Graduate School of Mathematical Sciences, The University of Tokyo, 3-8-1 Komaba Meguro-ku, 153-8914, TOKYO, JAPAN}
\thanks{email: kmotoko@ms.u-tokyo.ac.jp}
}

\begin{document}

\numberwithin{equation}{section}
\newtheorem{theorem}{Theorem}[section]
\newtheorem{proposition}[theorem]{Proposition}
\newtheorem{lemma}[theorem]{Lemma}
\newtheorem{corollary}[theorem]{Corollary}
\newtheorem{remark}{Remark}[section]
\newtheorem{definition}{Definition}[section]

\maketitle

\begin{abstract}
The Thompson group $V$ is a subgroup of the homeomorphism
group of the Cantor set $C$. Brin \cite{Brin} defined higher dimensional Thompson groups $nV$ as generalizations of $V$.
For each $n$, $nV$ is a subgroup of the homeomorphism group of $C^n$.
We prove that $nV$ has property $FA$, and especially the number of ends of $nV$ is equal to $1$. 
This is a generalization of the corresponding result of Farley \cite{Farley}, who studied the Thompson group $V=1V$.
\end{abstract}

\noindent \textbf{\textit{2010 Mathematics Subject Classification. }}Primary: 20F69, Secondary: 37E05.

\noindent \textbf{\textit{Keywords and phrases. }}The Thompson group, Higher dimensional Thompson groups, ends of group pairs, property FA.

\section{Intoduction}
Higher dimensional Thompson groups $nV$ were introduced by Brin in \cite{Brin} as generalizations of the Thompson group $V$.
The Thompson group $V$ is an infinite simple finitely presented group, which is described as a subgroup of the homeomorphism group of the Cantor set $C$.
Basic facts about $V$ are found in a paper by Cannon, Floyd and Parry \cite{CFP}.

Brin first studied the case of $n=2$ in detail, and showed that $V$ and $2V$ are not isomorphic (\cite{Brin}),
$2V$ is simple (\cite{Brin}) and $2V$ is finitely presented (\cite{Brin2005}).
These properties also hold true for general $nV$.
The simplicity of $nV$ was shown by Brin later in \cite{Brin2010}.
Bleak and Lanoue showed $n_1V$ and $n_2V$ are isomorphic if and only if $n_1=n_2$ in \cite{BL}.
Hennig and Matucci gave a finite presentation for each $nV$ (\cite{HM}).

In this paper we prove that for each $n$, $nV$ has property FA.
This result is the generalization of the corresponding result of Farley~\cite{Farley}, who studied $V$.

I thank my adviser, Takuya Sakasai for his support.
Koji Fujiwara, Tomohiko Ishida,
Tomohiro Fukaya and Masato Mimura gave insightful comments on the result.
Matthew G.\ Brin and Daniel S.\ Farley gave helpful comments and informed me related problems.
This work was supported by the Program for Leading Graduate Schools, MEXT, Japan.


\section{Higher dimensional Thompson groups $nV$}\label{nV}

In this section, we give the definition of higher dimensional Thompson groups according to Brin's paper \cite{Brin}.
The symbol $I$ denotes $[0, 1)$ throughout this paper.

An {\it $n$-dimensional rectangle} is defined inductively as follows.
First, $I^n$ is a rectangle.

If $R=[a_1, b_1)\times\cdots\times[a_i, b_i)\times\cdots\times[a_n, b_n)$ is a rectangle,
then for all $i\in\{1, \ldots, n \}$, the ``$i$-th left half'' and ``the $i$-th right half'' defined by
\begin{align}
R_{l,i}&=[a_1, b_1)\times\cdots\times[a_i, {(a_i+b_i)}/2)\times\cdots\times[a_n, b_n) \label{R_l}\\
R_{r,i}&=[a_1, b_1)\times\cdots\times[{(a_i+b_i)}/2, b_i)\times\cdots\times[a_n, b_n) \label{R_r}
\end{align}
are again rectangles.

Throughout this paper, $I_l$ denotes $[0, 1/2)\times I^{n-1}$.
Similarly, $I_r$ denotes $[1/2, 1)\times I^{n-1}$.

Let $R=[a_1, b_1)\times\cdots\times[a_i, b_i)\times\cdots\times[a_n, b_n)$ be a rectangle. 
A {\it corner} of $R$ is a point in $\cl(R)$, whose $i$-th coordinate is either $a_i$ or $b_i$.
Here $\cl(R)$ denotes the closure of $R$ in $\mathbb{R}^n$.
An {\it $n$-dimensional pattern} is a finite set of $n$-dimensional rectangles, with pairwise disjoint, non-empty interiors and whose union is $I^n$.
A {\it numbered pattern} is a pattern with a one-to-one
correspondence to $\{0, 1, \ldots , r-1\}$ where $r$ is the number of rectangles in the pattern.

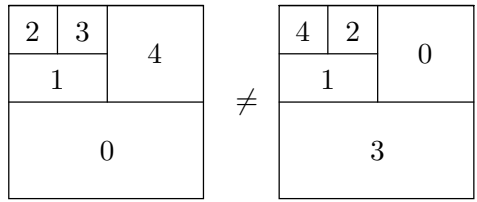
\begin{center}
{\unitlength 0.1in%
\begin{picture}( 24.0800, 10.0800)(  2.0000,-12.0800)%
%
\special{pn 8}%
\special{pa 200 200}%
\special{pa 1208 200}%
\special{pa 1208 1208}%
\special{pa 200 1208}%
\special{pa 200 200}%
\special{pa 1208 200}%
\special{fp}%
%
\special{pn 8}%
\special{pa 200 704}%
\special{pa 1208 704}%
\special{fp}%
%
\special{pn 8}%
\special{pa 710 704}%
\special{pa 710 200}%
\special{fp}%
\put(7.1000,-9.5600){\makebox(0,0){$0$}}%
%
\put(9.5600,-4.8000){\makebox(0,0)[lb]{}}%
%
\special{pn 8}%
\special{pa 710 452}%
\special{pa 200 452}%
\special{fp}%
%
\special{pn 8}%
\special{pa 452 200}%
\special{pa 452 452}%
\special{fp}%
\put(3.2300,-3.3400){\makebox(0,0){$2$}}%
\put(5.8100,-3.3400){\makebox(0,0){$3$}}%
\put(4.5200,-5.8600){\makebox(0,0){$1$}}%
\put(9.5600,-4.5200){\makebox(0,0){$4$}}%
\put(14.3200,-7.0400){\makebox(0,0){$\not=$}}%
%
\special{pn 8}%
\special{pa 1600 200}%
\special{pa 2608 200}%
\special{pa 2608 1208}%
\special{pa 1600 1208}%
\special{pa 1600 200}%
\special{pa 2608 200}%
\special{fp}%
%
\special{pn 8}%
\special{pa 1600 704}%
\special{pa 2608 704}%
\special{fp}%
%
\special{pn 8}%
\special{pa 2110 704}%
\special{pa 2110 200}%
\special{fp}%
\put(21.1000,-9.5600){\makebox(0,0){$3$}}%
%
\special{pn 8}%
\special{pa 2110 452}%
\special{pa 1600 452}%
\special{fp}%
%
\special{pn 8}%
\special{pa 1852 200}%
\special{pa 1852 452}%
\special{fp}%
\put(17.2300,-3.3400){\makebox(0,0){$4$}}%
\put(19.8100,-3.3400){\makebox(0,0){$2$}}%
\put(18.5200,-5.8600){\makebox(0,0){$1$}}%
\put(23.5600,-4.5200){\makebox(0,0){$0$}}%
\end{picture}}%
\end{center}

From now on, we will identify $n$-dimensional rectangle with a subset of $C^n$ and use the common symbol.
First we identify $I^n$ and $C^n$. $I$ denotes both $[0, 1)$ and $C$.
Let $R$ be a rectangle which is identified with a subset of $C^n$,
\begin{align}
R'=C^n\cap [{a'}_1, {b'}_1]\times\cdots\times[{a'}_i, {b'}_i]\times\cdots\times[{a'}_n, {b'}_n].
\end{align} 
Define rectangles $R_{l,i}$ and $R_{r,i}$ in the same way as we obtained (\ref{R_l}) and (\ref{R_r}).
These rectangles are identified respectively with the ``$i$-th left third'' and the ``$i$-th right third'' of $R'$, which is defined by
\begin{align}
C^n&\cap [{a'}_1, {b'}_1]\times\cdots\times[{a'}_i, {(2{a'}+{b'}_i)}/3]\times\cdots\times[{a'}_n, {b'}_n],\\
C^n&\cap [{a'}_1, {b'}_1]\times\cdots\times[{a'}_i, {({a'}+2{b'}_i)}/3]\times\cdots\times[{a'}_n, {b'}_n].
\end{align}
We proceed by induction.
In the same manner, every pattern describes a division of $C^n$.

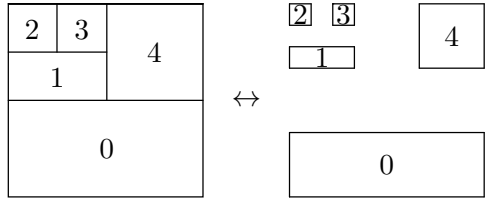
\begin{center}
{\unitlength 0.1in%
\begin{picture}( 24.6400, 10.1700)(  2.0000,-12.0800)%
%
\special{pn 8}%
\special{pa 200 200}%
\special{pa 1208 200}%
\special{pa 1208 1208}%
\special{pa 200 1208}%
\special{pa 200 200}%
\special{pa 1208 200}%
\special{fp}%
%
\special{pn 8}%
\special{pa 200 704}%
\special{pa 1208 704}%
\special{fp}%
%
\special{pn 8}%
\special{pa 710 704}%
\special{pa 710 200}%
\special{fp}%
\put(7.1000,-9.5600){\makebox(0,0){$0$}}%
%
\put(9.5600,-4.8000){\makebox(0,0)[lb]{}}%
%
\special{pn 8}%
\special{pa 710 452}%
\special{pa 200 452}%
\special{fp}%
%
\special{pn 8}%
\special{pa 452 200}%
\special{pa 452 452}%
\special{fp}%
\put(3.2300,-3.3400){\makebox(0,0){$2$}}%
\put(5.8100,-3.3400){\makebox(0,0){$3$}}%
\put(4.5200,-5.8600){\makebox(0,0){$1$}}%
\put(9.5600,-4.5200){\makebox(0,0){$4$}}%
\put(14.3200,-7.0400){\makebox(0,0){$\leftrightarrow$}}%
%
\special{pn 8}%
\special{pa 1656 1208}%
\special{pa 2664 1208}%
\special{fp}%
\special{pa 2664 1208}%
\special{pa 2664 872}%
\special{fp}%
\special{pa 2664 872}%
\special{pa 1656 872}%
\special{fp}%
\special{pa 1656 872}%
\special{pa 1656 1208}%
\special{fp}%
%
\special{pn 8}%
\special{pa 2328 536}%
\special{pa 2664 536}%
\special{fp}%
\special{pa 2664 536}%
\special{pa 2664 200}%
\special{fp}%
\special{pa 2664 200}%
\special{pa 2328 200}%
\special{fp}%
\special{pa 2328 200}%
\special{pa 2328 536}%
\special{fp}%
%
\special{pn 8}%
\special{pa 1992 536}%
\special{pa 1992 424}%
\special{fp}%
\special{pa 1992 424}%
\special{pa 1656 424}%
\special{fp}%
\special{pa 1656 424}%
\special{pa 1656 536}%
\special{fp}%
\special{pa 1656 536}%
\special{pa 1992 536}%
\special{fp}%
%
\special{pn 8}%
\special{pa 1656 312}%
\special{pa 1656 200}%
\special{fp}%
\special{pa 1656 200}%
\special{pa 1768 200}%
\special{fp}%
\special{pa 1768 200}%
\special{pa 1768 312}%
\special{fp}%
\special{pa 1768 312}%
\special{pa 1656 312}%
\special{fp}%
%
\special{pn 8}%
\special{pa 1880 200}%
\special{pa 1880 312}%
\special{fp}%
\special{pa 1880 312}%
\special{pa 1992 312}%
\special{fp}%
\special{pa 1992 312}%
\special{pa 1992 200}%
\special{fp}%
\special{pa 1992 200}%
\special{pa 1880 200}%
\special{fp}%
\special{pa 1880 200}%
\special{pa 1880 200}%
\special{fp}%
\put(17.1200,-2.5600){\makebox(0,0){$2$}}%
\put(19.3600,-2.5600){\makebox(0,0){$3$}}%
\put(18.2400,-4.8000){\makebox(0,0){$1$}}%
\put(24.9600,-3.6800){\makebox(0,0){$4$}}%
\put(21.6000,-10.4000){\makebox(0,0){$0$}}%
\end{picture}}%
\end{center}

We will construct a self-homeomorphism of $C^n$ from a pair of numbered patterns with the same number of rectangles.
Let $P=\{P_i\}_{0\leq i\leq r-1}$ and $Q=\{Q_i\}_{0\leq i\leq r-1}$ be numbered patterns.
We define $g(P,Q):I^n\to I^n$ which takes each $P_i$ onto $Q_i$ affinely so as to preserve the orientation.
Namely, the restriction of $g(P,Q)$ to each $P_i$ has the form $(x_1, \ldots, x_n)\mapsto (a_1 + {3}^{j_1}x_1, \ldots, a_n + {3}^{j_n}x_n)$
for some integers $j_1, \ldots, j_n$.
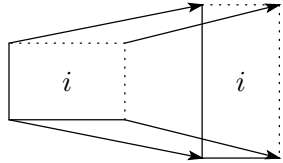
\begin{center}
{\unitlength 0.1in%
\begin{picture}( 14.0000,  8.4200)(  4.0000,-10.0000)%
%
\special{pn 8}%
\special{pa 400 400}%
\special{pa 400 800}%
\special{fp}%
\special{pa 400 800}%
\special{pa 1000 800}%
\special{fp}%
%
\special{pn 8}%
\special{pa 1000 800}%
\special{pa 1000 400}%
\special{dt 0.045}%
\special{pa 1000 400}%
\special{pa 400 400}%
\special{dt 0.045}%
%
\special{pn 8}%
\special{pa 1400 200}%
\special{pa 1800 200}%
\special{dt 0.045}%
\special{pa 1800 200}%
\special{pa 1800 1000}%
\special{dt 0.045}%
%
\special{pn 8}%
\special{pa 1800 1000}%
\special{pa 1400 1000}%
\special{fp}%
\special{pa 1400 1000}%
\special{pa 1400 200}%
\special{fp}%
%
\special{pn 8}%
\special{pa 400 400}%
\special{pa 1400 200}%
\special{fp}%
\special{sh 1}%
\special{pa 1400 200}%
\special{pa 1331 193}%
\special{pa 1348 210}%
\special{pa 1339 233}%
\special{pa 1400 200}%
\special{fp}%
%
\special{pn 8}%
\special{pa 1000 400}%
\special{pa 1800 200}%
\special{fp}%
\special{sh 1}%
\special{pa 1800 200}%
\special{pa 1730 197}%
\special{pa 1748 213}%
\special{pa 1740 236}%
\special{pa 1800 200}%
\special{fp}%
\special{pa 1000 800}%
\special{pa 1800 1000}%
\special{fp}%
\special{sh 1}%
\special{pa 1800 1000}%
\special{pa 1740 964}%
\special{pa 1748 987}%
\special{pa 1730 1003}%
\special{pa 1800 1000}%
\special{fp}%
\special{pa 400 800}%
\special{pa 1400 1000}%
\special{fp}%
\special{sh 1}%
\special{pa 1400 1000}%
\special{pa 1339 967}%
\special{pa 1348 990}%
\special{pa 1331 1007}%
\special{pa 1400 1000}%
\special{fp}%
\put(7.0000,-6.0000){\makebox(0,0){$i$}}%
\put(16.0000,-6.0000){\makebox(0,0){$i$}}%
\end{picture}}%
\end{center}

With the former identification of rectangles with subsets of $C^n$, above construction defines a self-homeomorphism of $C^n$.
We again write $g(P,Q)$ for this homeomorphism.  

When $n=2$, we illustrate $g(P,Q)$ as follows.
First we draw $P$ and $Q$ as divisions of $I^2$. 
Next we add an arrow from $P$ to $Q$, which indicates the domain and the range.

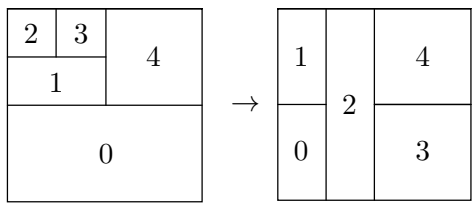
\begin{center}
{\unitlength 0.1in%
\begin{picture}( 24.0000, 10.0800)(  2.0000,-12.0800)%
%
\special{pn 8}%
\special{pa 200 200}%
\special{pa 1208 200}%
\special{pa 1208 1208}%
\special{pa 200 1208}%
\special{pa 200 200}%
\special{pa 1208 200}%
\special{fp}%
%
\special{pn 8}%
\special{pa 200 704}%
\special{pa 1208 704}%
\special{fp}%
%
\special{pn 8}%
\special{pa 710 704}%
\special{pa 710 200}%
\special{fp}%
\put(7.1000,-9.5600){\makebox(0,0){$0$}}%
%
\put(9.5600,-4.8000){\makebox(0,0)[lb]{}}%
%
\special{pn 8}%
\special{pa 710 452}%
\special{pa 200 452}%
\special{fp}%
%
\special{pn 8}%
\special{pa 452 200}%
\special{pa 452 452}%
\special{fp}%
\put(3.2300,-3.3400){\makebox(0,0){$2$}}%
\put(5.8100,-3.3400){\makebox(0,0){$3$}}%
\put(4.5200,-5.8600){\makebox(0,0){$1$}}%
\put(9.5600,-4.5200){\makebox(0,0){$4$}}%
\put(14.3200,-7.0400){\makebox(0,0){$\rightarrow$}}%
%
\special{pn 8}%
\special{pa 1600 200}%
\special{pa 2600 200}%
\special{pa 2600 1200}%
\special{pa 1600 1200}%
\special{pa 1600 200}%
\special{pa 2600 200}%
\special{fp}%
%
\special{pn 8}%
\special{pa 2100 200}%
\special{pa 2100 1200}%
\special{fp}%
%
\special{pn 8}%
\special{pa 2100 700}%
\special{pa 2600 700}%
\special{fp}%
\special{pa 2100 700}%
\special{pa 2600 700}%
\special{fp}%
%
\special{pn 8}%
\special{pa 1850 1200}%
\special{pa 1850 200}%
\special{fp}%
%
\special{pn 8}%
\special{pa 1600 700}%
\special{pa 1850 700}%
\special{fp}%
\put(17.2000,-4.5000){\makebox(0,0){$1$}}%
\put(17.2000,-9.4000){\makebox(0,0){$0$}}%
\put(19.7000,-7.0000){\makebox(0,0){$2$}}%
\put(23.5000,-4.5000){\makebox(0,0){$4$}}%
\put(23.5000,-9.4000){\makebox(0,0){$3$}}%
\end{picture}}%
\end{center}

The {\it $n$-dimensional Thompson group} $nV$ is the set of self-homeomorphisms of $C^n$ of the form $g(P, Q)$.
Every element of $nV$ is identified with a partially affine, partially orientation preserving bijection from $I^n$ to itself.

Next is an important property which will be used in later discussion.
\begin{theorem}[{Brin \cite{Brin2010}}]
For all $n\in \mathbb{N}$, $nV$ is simple.
\end{theorem}




\section{Ends of groups}\label{ends}
Let $\Gamma$ be a path-connected locally finite CW complex.
For a compact subset $K$, 
$\|\Gamma-K\|$ denotes the number of unbounded connected components of $\Gamma-K$.
The {\it number of ends of $\Gamma$}, $e(\Gamma)$, 
is defined to be the supremum of $\|\Gamma-K\|$ taken over all the compact subsets.

When $\Gamma$ is a graph, 
we equip $\Gamma$ with graph metric.
$B(m)$ denotes a ball of radius $m$ in $\Gamma$, based at some fixed vertex.
For simplicity, we ignore the dependence of $B(m)$ on the base point in notation.

Throughout this section, $G$ denotes a finitely generated group and $S$ denotes a finite generating set of $G$.
The {\it Cayley graph} $\Gamma_{G,S}$ is a graph whose vertex set is $G$,
and there is an oriented edge from $g\in G$ to $h\in G$ if some $s\in S$ satisfies $g\cdot s=h$. 
$G$ acts freely on $\Gamma_{G,S}$ from the left.

The {\it number of ends of} $G$, $e(G)$, is the number of ends of $\Gamma_{G,S}$.

\begin{theorem}[{cf.\ Geoghegan \cite[Corollary 13.5.12]{Geoghegan}}]\label{CW} 
Let $\Gamma$ be a path-connected locally finite CW complex on which $G$ acts freely.
Further suppose that the quotient space $\Gamma/G$ is a finite CW-complex.
Then $e(\Gamma)=e(G)$.
\end{theorem}

\begin{proposition}\label{FH}
\begin{itemize}
\item[$(1)$] $e(G)$ does not depend on the choice of $S$.
\item[$(2)$] $($The Freudenthal-Hopf Theorem\/$)$ $e(G)$ is $0$, $1$, $2$ or $\infty$. 
\item[$(3)$] $e(G)=0$ if and only if $G$ is finite.
\item[$(4)$] $e(G)=2$ if and only if $G$ has an infinite cyclic subgroup of finite index. 
\end{itemize}
\end{proposition}
 
The following result, Stallings' theorem, provides a group-theoretical characterization of the case where $e(G)\geq 2$.
\begin{theorem}[Stallings \cite{Stallings}, Bergman \cite{Bergman}]
$e(G)\geq 2$ if and only if $G$ has a structure of an amalgamated product or an HNN-extension on some finite subgroup.
\end{theorem}

In the light of this theorem, we can characterize the case of $e(G)=1$ in terms of group actions on trees.
We say that $G$ has {\it property FA} if every simplicial action of $G$ on a simplicial tree without edge-inversions has a fixed point. 
Here, a fixed point means $x\in T$ such that $g(x)=x$ for every $g\in G$. 

\begin{theorem}[Serre \cite{Serre}]
If an infinite group has propery FA, then $e(G)=1$.
\end{theorem}


\section{$nV$ has property FA}\label{e(nV)}

In this section, we prove that $nV$ has property FA, using a finite presentation of $nV$.
Throughout this section, $T$ denotes a simplicial tree. 
For $x, y\in T$, we write $[x:y]$ for the geodesic joining $x$ to $y$.
An action of a group on $T$ is assumed to be simplicial and to act without edge inversions.

Let $G$ be a group acting on $T$. 
Let $g \in G$.
If $\Fix(g)$ is non-empty, $g$ is said to be {\it elliptic}.
Otherwise, we say that $g$ is {\it hyperbolic}. 

The following proposition is a basic fact about group actions on trees.
\begin{proposition}[{Serre \cite{Serre}}]\label{Serre's lemma}
Let $G$ be a group acting on $T$. Let $g \in G$.
\begin{itemize}
\item[$(1)$] $\Fix(g)=\{\, x\in T\mid g(x)=x\,\}$ is either empty or a subtree of $T$.
\item[$(2)$] If $g$ is hyperbolic, $g$ acts on a unique simplicial line in $T$ by translation. This line is called the axis of $g$.
\item[$(3)$] $($Serre's lemma\/$)$ Assume that $G$ is generated by a finite set of elements $\{s_j\}_{1\leq i\leq m}$ such that 
every element and the multiplication of every two elements are elliptic. 
Then there is $x\in T$ which is fixed by every element of $G$.
\end{itemize}
\end{proposition} 

\begin{lemma}\label{commutativity}
Let $G$ be a group acting on $T$.
If $g$ and $h$ are elliptic and satisfy $gh=hg$, then $g$ and $h$ have a common fixed point.
\end{lemma}

\begin{proof}
Let $g$ and $h$ be elliptic elements which satisfy $gh=hg$.
Assume to the contrary that $g$ and $h$ do not have a common fixed point. 
Fix $y\in \Fix(h)$. 
Let $[y:x]$ be the shortest geodesic joining $y$ to $\Fix(g)$. 
The composition of $g^{-1}([y:x])$ and $[y:x]$ is $[y:g^{-1}(y)]$.
Now $g^{-1}(y)\in \Fix(h)$, because $h^{-1}g^{-1}(y)=g^{-1}h^{-1}(y)=g^{-1}(y)$.
By Lemma~\ref{Serre's lemma} $(1)$, $[y:g^{-1}(y)]\subset \Fix(h)$.
Therefore $x\in \Fix(h)$.
This contradicts our assumption.
\end{proof}

We define $X_{1,0}$, $X_{d',0}$, $C_{d', 0}$, $\pi_0$, $\overline{\pi}_0 \in nV$ $(2\leq d'\leq n)$ as shown in the following figure.
For $i\geq 1$, $X_{d,i}$ $(1\leq d\leq n)$ is defined inductively.
On $I_r$, $X_{d,i}$ restricts to the identity.
For $x\in I_l$, we write $x=(x_1, x_2)$ where $x_1\in [0,1/2)$ and $x_2\in I^{n-1}$.
We define $\phi:I_l\to I^n$ by $\phi(x_1, x_2)=(2x_1, x_2)$.  
On $I_l$, $X_{d,i}=X_{d, i-1}\phi$.
Similarly, $C_{d', i}$, $\pi_i$ and $\overline{\pi}_{i}$ restricts to the identity on $I_r$
and $C_{d',i-1}\phi$, $\pi_{i-1}\phi$ and $\overline{\pi}_{i-1}\phi$ respectively on $I_l$.

\begin{center}
\input{generators.tex}
\end{center}

\begin{theorem}[{Hennig and Matucci \cite[Theorem 23]{HM}}]\label{presentation}
Let
\begin{align}
\Sigma = \{X_{d,i}, C_{d', i}, \pi_i, \overline{\pi}_i\}_{1\leq d\leq n,\ 2\leq d'\leq n,\ i\geq 0}.
\end{align}
\begin{itemize}
\item[$(1)$] $\Sigma$ is a generating set of $nV$.\\
\item[$(2)$] The elements of $\Sigma$ satisfy the following relations.\\
\begin{align}
X_{d'',j} X_{d,i} &= X_{d,i} X_{d'', j+1}  &(i<j, 1\leq d, d''\leq n) \label{relation1}\\
C_{d',j}X_{d,i} &= X_{d,i} C_{d',j+1}  &(i<j, 1\leq d\leq n, 2\leq d'\leq n) \label{relation1_C}\\
Y_j X_{d,i} &= X_{d,i} Y_{j+1} &(i<j, Y\in\{\,\pi, \overline{\pi}\,\}, 1\leq d\leq n) \label{relation1_pi}\\
\pi_j X_{d,i} &= X_{d,i}\pi_j &(i>j+1, 1\leq d\leq n) \label{relation2}\\
\pi_j C_{d',i} &= C_{d',i}\pi_j &(i>j+1, 2\leq d'\leq n) \label{relation2_C}\\
\pi_j \pi_i &=\pi_i \pi_j &(|i-j|>2) \label{relation2_pi}\\
\overline{\pi}_j \pi_i &=\pi_i \overline{\pi}_j &(j>i+1) \label{relation2_pibar}\\
\overline{\pi}_i X_{1,i} &= \pi_i \overline{\pi}_{i+1} &(i\geq 0) \label{relation3}\\
C_{d',i} X_{1,i}&= X_{d',i}C_{d',i+2} \pi_{i+1} &(i\geq 0, 2\leq d'\leq n) \label{relation4}\\
\pi_i X_{d,i} &= X_{d,i+1} \pi_i \pi_{i+1} &(i\geq 0, 1\leq d\leq n) \label{relation5}
\end{align}
\end{itemize}
\end{theorem}

\begin{corollary}
Let
\begin{align}
 S = \{\,X_{d,1}, X_{d,1}{(X_{d,0})}^{-1}, C_{d',2}, \pi_0, \pi_3, \overline{\pi}_3\,\}_{1\leq d\leq n, 2\leq d'\leq n}. \label{generators}
\end{align}
This is a generating set of $nV$.
\end{corollary}

\begin{center}
{\unitlength 0.1in%
\begin{picture}( 33.5200, 18.7000)(  0.3000,-20.7000)%
\put(7.6800,-6.7900){\makebox(0,0){$X_{1,1}{(X_{1,0})}^{-1}=$}}%
%
\special{pn 8}%
\special{pa 1460 360}%
\special{pa 2100 360}%
\special{pa 2100 999}%
\special{pa 1460 999}%
\special{pa 1460 360}%
\special{pa 2100 360}%
\special{fp}%
%
\special{pn 8}%
\special{pa 1460 360}%
\special{pa 1620 200}%
\special{fp}%
\special{pa 1620 200}%
\special{pa 2260 200}%
\special{fp}%
\special{pa 2260 200}%
\special{pa 2100 360}%
\special{fp}%
\special{pa 2260 200}%
\special{pa 2260 839}%
\special{fp}%
\special{pa 2260 839}%
\special{pa 2100 999}%
\special{fp}%
%
\special{pn 8}%
\special{pa 1780 999}%
\special{pa 1780 999}%
\special{fp}%
\special{pa 1780 999}%
\special{pa 1780 360}%
\special{fp}%
\special{pa 1780 360}%
\special{pa 1940 200}%
\special{fp}%
\special{pa 2100 200}%
\special{pa 1940 360}%
\special{fp}%
\special{pa 1940 360}%
\special{pa 1940 999}%
\special{fp}%
\special{pa 1620 999}%
\special{pa 1620 360}%
\special{fp}%
\special{pa 1620 360}%
\special{pa 1780 200}%
\special{fp}%
\special{pa 2420 679}%
\special{pa 2420 679}%
\special{fp}%
\put(24.2000,-6.7900){\makebox(0,0){$\rightarrow$}}%
%
\special{pn 8}%
\special{pa 2581 360}%
\special{pa 3221 360}%
\special{pa 3221 999}%
\special{pa 2581 999}%
\special{pa 2581 360}%
\special{pa 3221 360}%
\special{fp}%
%
\special{pn 8}%
\special{pa 3221 999}%
\special{pa 3381 839}%
\special{fp}%
\special{pa 3381 839}%
\special{pa 3381 200}%
\special{fp}%
\special{pa 3381 200}%
\special{pa 3221 360}%
\special{fp}%
\special{pa 3381 200}%
\special{pa 2741 200}%
\special{fp}%
\special{pa 2741 200}%
\special{pa 2581 360}%
\special{fp}%
\special{pa 2901 360}%
\special{pa 2901 999}%
\special{fp}%
\special{pa 2741 999}%
\special{pa 2741 360}%
\special{fp}%
\special{pa 2741 360}%
\special{pa 2901 200}%
\special{fp}%
\special{pa 3061 200}%
\special{pa 2901 360}%
\special{fp}%
\special{pa 2821 360}%
\special{pa 2821 999}%
\special{fp}%
\special{pa 2821 360}%
\special{pa 2981 200}%
\special{fp}%
\put(15.4000,-6.7900){\makebox(0,0){$0$}}%
\put(17.0000,-6.7900){\makebox(0,0){$1$}}%
\put(18.6000,-6.7900){\makebox(0,0){$2$}}%
\put(20.2000,-6.7900){\makebox(0,0){$3$}}%
\put(26.6100,-6.7900){\makebox(0,0){$0$}}%
\put(27.8100,-6.7900){\makebox(0,0){$1$}}%
\put(28.6100,-6.7900){\makebox(0,0){$2$}}%
\put(30.6100,-6.7900){\makebox(0,0){$3$}}%
\put(7.6000,-16.8400){\makebox(0,0){$X_{d',1}{(X_{d',0})}^{-1}=$}}%
%
\special{pn 8}%
\special{pa 1461 1365}%
\special{pa 2101 1365}%
\special{pa 2101 2003}%
\special{pa 1461 2003}%
\special{pa 1461 1365}%
\special{pa 2101 1365}%
\special{fp}%
%
\special{pn 8}%
\special{pa 1461 1365}%
\special{pa 1621 1205}%
\special{fp}%
\special{pa 1621 1205}%
\special{pa 2261 1205}%
\special{fp}%
\special{pa 2261 1205}%
\special{pa 2101 1365}%
\special{fp}%
\special{pa 2261 1205}%
\special{pa 2261 1844}%
\special{fp}%
\special{pa 2261 1844}%
\special{pa 2101 2003}%
\special{fp}%
\put(24.2100,-16.8400){\makebox(0,0){$\rightarrow$}}%
%
\special{pn 8}%
\special{pa 2582 1365}%
\special{pa 3222 1365}%
\special{pa 3222 2003}%
\special{pa 2582 2003}%
\special{pa 2582 1365}%
\special{pa 3222 1365}%
\special{fp}%
\put(15.4100,-16.8400){\makebox(0,0){$0$}}%
\put(17.0100,-16.8400){\makebox(0,0){$1$}}%
\put(26.6200,-16.8400){\makebox(0,0){$0$}}%
\put(30.6200,-16.8400){\makebox(0,0){$3$}}%
%
\special{pn 8}%
\special{pa 1781 1365}%
\special{pa 1781 2003}%
\special{fp}%
\special{pa 1621 2003}%
\special{pa 1621 1365}%
\special{fp}%
\special{pa 1621 1365}%
\special{pa 1781 1205}%
\special{fp}%
\special{pa 1781 1205}%
\special{pa 1941 1205}%
\special{fp}%
\special{pa 1941 1205}%
\special{pa 1781 1365}%
\special{fp}%
\special{pa 1941 1205}%
\special{pa 1781 1205}%
\special{fp}%
%
\special{pn 8}%
\special{pa 2101 1684}%
\special{pa 1941 1684}%
\special{fp}%
\special{pa 1941 1684}%
\special{pa 1781 1684}%
\special{fp}%
\put(19.4100,-18.4400){\makebox(0,0){$2$}}%
\put(19.4100,-15.2400){\makebox(0,0){$3$}}%
%
\special{pn 8}%
\special{pa 2582 1365}%
\special{pa 2742 1205}%
\special{fp}%
\special{pa 2742 1205}%
\special{pa 3382 1205}%
\special{fp}%
\special{pa 3382 1205}%
\special{pa 3222 1365}%
\special{fp}%
\special{pa 3382 1205}%
\special{pa 3382 1844}%
\special{fp}%
\special{pa 3382 1844}%
\special{pa 3222 2003}%
\special{fp}%
%
\special{pn 8}%
\special{pa 3062 1205}%
\special{pa 2902 1365}%
\special{fp}%
\special{pa 2902 1365}%
\special{pa 2902 2003}%
\special{fp}%
%
\special{pn 8}%
\special{pa 2742 2003}%
\special{pa 2742 1365}%
\special{fp}%
\special{pa 2742 1365}%
\special{pa 2902 1205}%
\special{fp}%
%
\special{pn 8}%
\special{pa 2902 1684}%
\special{pa 2742 1684}%
\special{fp}%
\put(28.2200,-18.4400){\makebox(0,0){$1$}}%
\put(28.2200,-15.2400){\makebox(0,0){$2$}}%
%
\special{pn 8}%
\special{pa 2101 2003}%
\special{pa 2261 2003}%
\special{fp}%
\special{sh 1}%
\special{pa 2261 2003}%
\special{pa 2194 1983}%
\special{pa 2208 2003}%
\special{pa 2194 2023}%
\special{pa 2261 2003}%
\special{fp}%
\special{pa 1461 1365}%
\special{pa 1461 1205}%
\special{fp}%
\special{sh 1}%
\special{pa 1461 1205}%
\special{pa 1441 1272}%
\special{pa 1461 1258}%
\special{pa 1481 1272}%
\special{pa 1461 1205}%
\special{fp}%
\put(23.5700,-19.5600){\makebox(0,0){$1$}}%
\put(13.3300,-12.0500){\makebox(0,0){$d$}}%
%
\special{pn 8}%
\special{pa 2100 999}%
\special{pa 2260 999}%
\special{fp}%
\special{sh 1}%
\special{pa 2260 999}%
\special{pa 2193 979}%
\special{pa 2207 999}%
\special{pa 2193 1019}%
\special{pa 2260 999}%
\special{fp}%
\special{pa 1460 360}%
\special{pa 1460 200}%
\special{fp}%
\special{sh 1}%
\special{pa 1460 200}%
\special{pa 1440 267}%
\special{pa 1460 253}%
\special{pa 1480 267}%
\special{pa 1460 200}%
\special{fp}%
\put(23.5600,-9.5900){\makebox(0,0){$1$}}%
\end{picture}}%
\end{center} 

\begin{proof}
Let $\langle S\rangle$ denote a subgroup generated by $S$. 
$X_{d,0}\in \langle S\rangle$.
For $i\geq 2$, the relation~(\ref{relation1}) shows that $X_{d,i}={(X_{d,0})}^{-(i-1)}X_{d,1}{(X_{d,0})}^{i-1}\in \langle S\rangle$.

Similarly, the relation~(\ref{relation1_pi}) shows that $Y_{i}={(X_{d,0})}^{-(i-3)}Y_3{(X_{d,0})}^{i-3}\in \langle S\rangle$ for $i\geq 1$, where $Y$ is $\pi$ or $\overline{\pi}$.
By the relation~(\ref{relation3}), $\overline{\pi}_0\in \langle S\rangle$.

The relation~(\ref{relation1_C}) shows that $C_{d',i}={(X_{d,0})}^{-(i-2)}C_{d',2}{(X_{d,0})}^{i-2}\in \langle S\rangle$ for $i\geq 1$.
By the relation~(\ref{relation4}), $C_{d',0}\in \langle S\rangle$.
\end{proof}

The next lemma is a generalization of Lemma~$4.2$ in \cite{Farley}.
\begin{lemma}\label{small element}
Let $g\in nV$ which acts identically on some rectangle. 
For any action of $nV$ on a tree $T$, $g$ is elliptic.
\end{lemma}
 
\begin{proof}
Let $g\in nV$ be an element with a rectangle $R$ on which $g$ acts as the identity. 
Assume to the contrary that $g$ is hyperbolic. 
We write $l_g$ for the axis of $g$.
Let
\begin{align}
H_g=\{\,h\in nV \mid \supp(h)\subseteq R\,\}\cong nV.
\end{align}

For every $h\in H_g$, $hg=gh$ and $g$ acts on $h(l_g)$ as a translation. 
By the uniqueness of the axis, $h(l_g)=l_g$.
Restricting the action of $h$ on $l_g$,
we regard $h$ as an element of the infinite dihedral group $D_{\infty}$. 
In this way we obtain a homomorphism 
$\Phi:H_g\to D_{\infty}$.
By the simplicity of $H_g$, $\ker\Phi$ is $H_g$ or the trivial subgroup.
We claim that $\ker\Phi$ is not trivial.
Indeed, $H_g$ has the subgroup which is isomorphic to the Thompson group $F$.
$\ker\Phi$ contains the commutator subgroup of $H_g$, because every proper quotient of $F$ is abelian (\cite{CFP}).
Hence $\ker\Phi=H_g$.

There is $k\in nV$ such that $k\cdot \supp(g)\subseteq R$.
For this $k$, $kgk^{-1}\in H_g$.
Therefore, $kgk^{-1}$ is elliptic, which contradicts our assumption that $g$ is hyperbolic.
\end{proof} 

The following theorem is the main result.
\begin{theorem}\label{main theorem}
$nV$ has property FA. Especially, $e(nV)=1$.
\end{theorem}

\begin{proof}
Let $S$ be the generating set of (\ref{generators}).
By Serre's lemma, 
it is enough to show that every element and the product of every two elements of $S$ are elliptic.
By Lemma~\ref{small element}, every element of $S$ is elliptic. 

$S=S_1\cup S_2$, where
\begin{align} 
S_1&=\{X_{d,1}, C_{d', 2}, \pi_3, \overline{\pi}_3\}_{1\leq d\leq n, 2\leq d'\leq n},\\
S_2&=\{X_{d,1}{(X_{d,0})}^{-1}, \pi_0\}_{1\leq d\leq n}.
\end{align}
Every element of $S_1$ acts as the identity on $I_r$.
Every element of $S_2$ acts as the identity on the ``left quarter'' of the unit cube, $[0, 1/4)\times I^{n-1}$. 
Therefore, Lemma~\ref{small element} shows that the product of every two elements in $S_i$ ($i=1, 2$) is elliptic. 

Next we consider 
${S_1}'=\{C_{d',2}, \pi_3, \overline{\pi}_3\}_{2\leq d'\leq n}\subset S_1$.
The relations~(\ref{relation1}), (\ref{relation1_C}) and (\ref{relation1_pi}) imply that $X_{d,1}{(X_{d,0})}^{-1}$ and $Z\in {S_1}'$ are commutative. 
In fact,
\begin{align*} 
X_{d,1}{(X_{d,0})}^{-1}Z{(X_{d,1}{(X_{d,0})}^{-1})}^{-1}Z^{-1}=(X_{d,1}({X_{d,0}}^{-1}Z{X}_{d,0}){X_{d,1}}^{-1}){Z}^{-1}
=1.
\end{align*}
The relations~(\ref{relation2_C}), (\ref{relation2_pi}) and (\ref{relation2_pibar}) say that $\pi_0$ and the elements of ${S_1}'$ are commutative. 
Thus Lemma~\ref{commutativity} shows the products of $Z\in {S_1}'$ and $X_{d,1}{(X_{d,0})}^{-1}$ or $\pi_0$ are elliptic. 

The rest of the proof is to show the following lemma.
\end{proof}

\begin{lemma}
For every $d, d''\in \{1, \ldots n\}$, 
\begin{itemize}
\item [$(1)$] $X_{d,1}$ and $X_{d'',1}{(X_{d'',0})}^{-1}$ have a common fixed point.
\item [$(2)$] $X_{d,1}$ and $\pi_0$ have a common fixed point. 
\end{itemize}
\end{lemma}
\begin{proof}

$(1)$ $X_{d'',1}$ and $X_{d,2}$ act identically on $I_r$.
By Lemma~\ref{small element} and Serre's lemma,
there exists $y\in T$ which is fixed by $X_{d'',1}$ and $X_{d,2}$.

To obtain a contradiction, suppose $X_{d,1}$ and $X_{d'',1}{(X_{d'',0})}^{-1}$ do not have a common fixed point.
Take the shortest geodesic $[y:x]$ joining $y$ to $\Fix(X_{d'',1}{(X_{d'',0})}^{-1})$. 
The composition of $(X_{d'',1}{(X_{d'',0})}^{-1})^{-1}[y:x]$ and $[y:x]$ is $[y:X_{d'',0}(y)]$. 
By the relation~(\ref{relation1}), 
\begin{align*}
X_{d,1}X_{d'',0}(y)=X_{d'',0}{(X_{d'',0})}^{-1}X_{d,1}X_{d'',0}(y)
=X_{d'',0}X_{d,2}(y)
=X_{d'',0}(y).
\end{align*}
Therefore, $[y:X_{d'',0}(y)]\subset \Fix(X_{d,1})$ and $x\in \Fix(X_{d,1})$. 
This contradicts our assumption.

$(2)$ We first show that $\pi_0$ and $X_{d,0}$ have a common fixed point. 
To obtain a contradiction, suppose $\pi_0$ and $X_{d,0}$ do not have a common fixed point.
By $(1)$, $\Fix(X_{d,0})$ is not empty.
We consider a new element ${(X_{d,0})}^{-1}X_{d,1}$, which acts as the identity on the left one-eighth of $I^n$.
Since $\pi_0$ and $\pi_1$ also act as the identity on this rectangle, we can take $y$ as a common fixed point of ${(X_{d,0})}^{-1}X_{d,1}$, $\pi_0$ and $\pi_1$.
There is the shortest geodesic $[y:x]$ joining $y$ to $\Fix(X_{d,0})$.
The composition of $[y:x]$ and $X_{d,0}([x:y])$ is $[y:X_{d,0}(y)]$. 
By the relation~(\ref{relation5}),
\begin{align*}
\pi_0 X_{d,0}(y)=X_{d,1}\pi_0\pi_1(y)
=X_{d,1}(y)
=X_{d,1}{({(X_{d,0})}^{-1}X_{d,1})}^{-1}(y)
=X_{d,0}(y).
\end{align*}
Therefore, $[y:X_{d,0}(y)]\subset \Fix(\pi_0)$ and $x\in \Fix(\pi_0)$.
This contradicts our assumption.

We consider a subgroup generated by $\{\pi_0, {(X_{d,0})}^{-1}X_{d,1}, X_{d,0}\}$.
By Serre's lemma, this subgroup has a fixed point.
Therefore, $\Fix(\pi_0)\cap \Fix(X_{d,1})$ is not empty.
\end{proof}



\begin{thebibliography}{99}
  \bibitem{Bergman} 
    G.\ M.\ Bergman, \textit{On groups acting on locally finite graphs}, Ann.\ of Math.\ (2) {\bf 88}, 335--340, 1968.
  \bibitem{BL}
    C.\ Bleak and D.\ Lanoue, \textit{A family of non-isomorphism results}, Geom.\ Dedicata, {\bf 146}, 21--26, 2010.
  \bibitem{Brin}
    M.\ G.\ Brin, \textit{Higher dimensional Thompson groups}, Geom.\ Dedicata, {\bf 108}, 163--192, 2004.
  \bibitem{Brin2005}
    M.\ G.\ Brin, \textit{Presentations of higher dimensional Thompson groups}, J.\ Algebra {\bf 284}, 520--558, 2005.
  \bibitem{Brin2010}
	M.\ G.\ Brin, \textit{On the baker's map and the simplicity of the higher dimensional Thompson groups $nV$}, Publ.\ Mat.\ {\bf 54}, 433--439, 2010.
  \bibitem{CFP}
	J.\ W.\ Cannon, W.\ J.\ Floyd, and W.\ R.\ Parry, \textit{Introductory notes on Richard Thompson's groups}, Enseign.\ Math.\ (2) {\bf 42}, 215--256, 1996.
  \bibitem{Farley}
    D.\ S.\ Farley, \textit{A proof that Thompson's groups have infinitely many relative ends}, J.\ Group Theory {\bf 14}, 649--656, 2011.
  \bibitem{Geoghegan}
    R.\ Geoghegan, Topological methods in group theory, Springer GTM, 2008.
  \bibitem{HM}
    J.\ Hennig and F.\ Matucci, \textit{Presentations for the higher-dimensional Thompson groups $nV$}, Pacific J.\ Math.\, {\bf 257}, 53--74, 2012.
  \bibitem{Serre}
    J-P.\ Serre, Trees, Springer, 1980.
   \bibitem{Stallings}
    J.\ R.\ Stallings, \textit{On torsion-free groups with infinitely many ends}, Ann.\ of Math.\ (2) {\bf 88}, 312--334, 1968.
\end{thebibliography}
\end{document}